\documentclass{amsart}

\usepackage{hyperref, amssymb}
\usepackage{color}
\usepackage[all]{xy}

\theoremstyle{plain}
\newtheorem{theorem}{Theorem}[section]
\newtheorem{proposition}[theorem]{Proposition}

\newtheorem{corollary}[theorem]{Corollary}

\theoremstyle{definition}

\newcommand{\f}{\varphi}

\newcommand{\CC}{\mathbb C}
\newcommand{\PP}{\mathbb P}

\newcommand{\MM}{\mathbf M}

\newcommand{\E}{{\mathcal E}}
\newcommand{\F}{{\mathcal F}}

\def\O{\mathcal O}

\newcommand{\Ker}{{\mathcal Ker}}
\newcommand{\Coker}{{\mathcal Coker}}
\newcommand{\Image}{{\mathcal Im}}

\newcommand{\Aut}{\operatorname{Aut}}

\newcommand{\Grass}{\operatorname{Grass}}

\newcommand{\Hom}{\operatorname{Hom}}
\def\H{\operatorname{H}}
\newcommand{\M}{\operatorname{M}}

\newcommand{\Poly}{\operatorname{P}}

\newcommand{\spann}{\operatorname{span}}
\newcommand{\Sym}{\operatorname{S}}

\newcommand{\tensor}{\otimes}

\newcommand{\lra}{\longrightarrow}

\begin{document}

\title[Moduli of sheaves supported on curves of genus $2$ in $\PP^1 \times \PP^1$]
{On the geometry of the moduli space of sheaves supported on curves of genus two in a quadric surface}

\author{Mario Maican}
\address{Institute of Mathematics of the Romanian Academy, Calea Grivitei 21, Bucharest 010702, Romania}

\email{maican@imar.ro}

\begin{abstract}
We study the moduli space of stable sheaves of Euler characteristic $2$,
supported on curves of arithmetic genus $2$ contained in a smooth quadric surface.
We show that this moduli space is rational.
We compute its Betti numbers and we give a classification of the stable sheaves involving locally free resolutions.
\end{abstract}


\maketitle

\section{Introduction}
\label{introduction}

Consider the quadric surface $\PP^1 \times \PP^1$ defined over $\CC$ with fixed polarization $\O(1, 1)$.
The Hilbert polynomial of a coherent algebraic sheaf $\F$ on $\PP^1 \times \PP^1$ is a polynomial $P$ in two variables with rational
coefficients which satisfies $P(m, n) = \chi(\F(m, n))$ for all integers $m$, $n$.
For a given $P$ we denote by $\M(P)$ the moduli space of sheaves on $\PP^1 \times \PP^1$
that are semi-stable relative to the fixed polarization and that have Hilbert polynomial $P$.
In this paper we will be concerned with the geometry of $\MM = \M(3m + 2n + 2)$.
The sheaves in $\MM$ have Euler characteristic $2$ and are supported on curves of bidegree $(2, 3)$.
According to \cite{lepotier}, $\MM$ is an irreducible smooth projective variety of dimension $13$.
In the following theorem we classify the sheaves in $\MM$.

\begin{theorem}
\label{main_theorem}
We have a decomposition of $\MM$ into an open subvariety $\MM_0$,
a closed smooth irreducible subvariety $\MM_1$ of codimension $1$, and a closed smooth irreducible subvariety $\MM_2$ of codimension $2$.
The subvarieties are defined as follows: $\MM_0 \subset \MM$ is the subset of sheaves $\F$ having a resolution of the form
\[
0 \lra \O(-1, -2) \oplus \O(-1, -1) \overset{\f}{\lra} 2\O \lra \F \lra 0
\]
with $\f_{12}$ and $\f_{22}$ linearly independent; $\MM_1 \subset \MM$ is the subset of sheaves $\F$ having a resolution of the form
\[
0 \lra \O(-2, -1) \oplus \O(-1, -2) \overset{\f}{\lra} \O(-1, -1) \oplus \O(0, 1) \lra \F \lra 0
\]
with $\f_{11} \neq 0$, $\f_{12} \neq 0$;
$\MM_2 \subset \MM$ is the subset of twisted structure sheaves $\O_C(1, 0)$ for a curve $C \subset \PP^1 \times \PP^1$ of bidegree $(2, 3)$.

The subvariety $\MM_1$ is isomorphic to the universal curve of bidegree $(2, 3)$ and is the Brill-Noether locus of sheaves $\F$
satisfying $\H^0(\F(0, -1)) \neq 0$ (for $\F \in \MM_1$ we have $\H^0(\F(0, -1)) \simeq \CC$);
$\MM_2$ is isomorphic to $\PP^{11}$ and is the Brill-Noether locus of sheaves $\F$ satisfying $\H^0(\F(-1, 0)) \neq 0$ (for $\F \in \MM_2$ we have
$\H^0(\F(-1, 0)) \simeq \CC$).
\end{theorem}

\noindent
The proof of the above theorem takes up Section \ref{classification} and uses a spectral sequence, discovered in \cite{buchdahl},
which is similar to the Beilinson spectral sequence on the projective space. We introduce this spectral sequence in Section \ref{preliminaries}.
As a consequence of Theorem \ref{main_theorem}, we will show in Section \ref{homology} that $\MM_0$ is obtained from a $\PP^9$-bundle over the
Grassmannian of planes in $\CC^4$ by removing two disjoint subvarieties isomorphic to $\PP^1$, respectively, to $\PP^1 \times \PP^1$.
This allows us to compute the Betti numbers of $\MM$.
The varieties $X$ occurring in this paper will have no odd homology, so it is convenient to work with the Poincar\'e polynomial
\[
\Poly(X)(\xi) = \sum_{i \ge 0} \dim_{\mathbb Q}^{} \H^i(X, {\mathbb Q}) \xi^{i/2}.
\]

\begin{theorem}
\label{poincare_polynomial}
The Poincar\'e polynomial of $\MM$ is
\[
\xi^{13} + 3 \xi^{12} + 8 \xi^{11} + 10 \xi^{10} + 11 \xi^9 + 11 \xi^8 + 11 \xi^7 + 11 \xi^6 + 11 \xi^5 + 11 \xi^4 + 10 \xi^3 + 8 \xi^2 + 3 \xi + 1.
\]
The Euler characteristic of $\MM$ is $110$.
\end{theorem}

\noindent
Remarkably, $\Poly(\MM)$ coincides with $\Poly(\M(3m + 2n + 1))$, which was computed in \cite[Section 9.2]{choi_katz_klemm}
and in \cite{genus_two} by different methods.


\section{Preliminaries}
\label{preliminaries}

Let $\F$ be a coherent sheaf on $\PP^1 \times \PP^1$ with support of dimension $1$.
According to \cite[Lemma 1]{buchdahl}, there is a spectral sequence converging to $\F$, whose $E_1$-term has display diagram
\begin{equation}
\label{E_1}
\xymatrix
{
\H^1(\F(-1, -1)) \tensor \O(-1, -1) = E_1^{-2, 1} \ar[r]^-{\f_1} & E_1^{-1, 1} \ar[r]^-{\f_2} & E_1^{0, 1} = \H^1(\F) \tensor \O \\
\H^0(\F(-1, -1)) \tensor \O(-1, -1) = E_1^{-2, 0} \ar[r]^-{\f_3} & E_1^{-1, 0} \ar[r]^-{\f_4} & E_1^{0, 0} = \H^0(\F) \tensor \O
}
\end{equation}
In addition, we have the exact sequences
\begin{equation}
\label{E_1^{-1,0}}
\H^0(\F(0, -1)) \tensor \O(0, -1) \lra E_1^{-1, 0} \lra \H^0(\F(-1, 0)) \tensor \O(-1, 0),
\end{equation}
\begin{equation}
\label{E_1^{-1,1}}
\H^1(\F(0, -1)) \tensor \O(0, -1) \lra E_1^{-1, 1} \lra \H^1(\F(-1, 0)) \tensor \O(-1, 0).
\end{equation}
The convergence of the spectral sequence implies that $\f_2$ is surjective and that we have the exact sequence
\begin{equation}
\label{convergence}
0 \lra \Ker(\f_1) \stackrel{\f_5}{\lra} \Coker(\f_4) \lra \F \lra \Ker(\f_2)/\Image(\f_1) \lra 0.
\end{equation}

Let $\E$ be a semi-stable sheaf on $\PP^1 \times \PP^1$ with Hilbert polynomial $P_{\E}(m,n) = rm + n + 1$.
According to \cite[Proposition 11]{ballico_huh}, $\E$ has resolution
\begin{equation}
\label{(r,1)}
0 \lra \O(-1, -r) \lra \O \lra \E \lra 0.
\end{equation}
Let $\E$ be a semi-stable sheaf on $\PP^1 \times \PP^1$ with $P_{\E}(m,n) = m + sn + 1$.
Then $\E$ has resolution
\begin{equation}
\label{(1,s)}
0 \lra \O(-s, -1) \lra \O \lra \E \lra 0.
\end{equation}

We fix vector spaces $V_1$ and $V_2$ over $\CC$ of dimension $2$ and we make the identification
$\PP^1 \times \PP^1 = \PP(V_1) \times \PP(V_2)$. We fix bases $\{ x, y \}$ of $V_1^*$ and $\{ z, w \}$ of $V_2^*$.


\section{Classification of sheaves}
\label{classification}

\begin{proposition}
\label{vanishing}
Let $\F$ give a point in $\MM$. Then
\begin{enumerate}
\item[(i)] $\H^0(\F(-1, -1)) = 0$;
\item[(ii)] $\H^1(\F) = 0$;
\item[(iii)] $\H^0(\F(-1,0)) \neq 0$ if and only if $\F \simeq \O_C(1, 0)$ for a curve $C \subset \PP^1 \times \PP^1$ of bidegree $(2, 3)$;
\item[(iv)] $\H^0(\F(0, -1)) \neq 0$ if and only if $\F$ is a non-split extension of the form
\[
0 \lra \O_C(0, 1) \lra \F \lra \CC_p \lra 0
\]
for a curve $C \subset \PP^1 \times \PP^1$ of bidegree $(2, 3)$ and a point $p \in C$, if and only if $\F$ has a resolution
\begin{equation}
\label{M_1}
0 \lra \O(-2, -1) \oplus \O(-1, -2) \overset{\f}{\lra} \O(-1, -1) \oplus \O(0, 1) \lra \F \lra 0
\end{equation}
with $\f_{11} \neq 0$, $\f_{12} \neq 0$.
\end{enumerate}
\end{proposition}

\begin{proof}
Parts (i) and (ii) follow directly from part (i), respectively, from part (ii) of \cite[Proposition 3.3]{genus_two}.

\medskip

\noindent
(iii) Assume that $\H^0(\F(-1,0)) \neq 0$. There is a curve $C$ and an injective morphism $\O_C \to \F(-1, 0)$
(consult the proof of \cite[Proposition 3.3]{genus_two}).
From Table 1 in the proof of \cite[Proposition 3.5]{genus_two} we see that the only case
in which $\O_C(1, 0)$ does not destabilize $\F$ is when $\deg (C) = (2, 3)$. Thus $\F \simeq \O_C(1, 0)$.
Conversely, given any curve $C \subset \PP^1 \times \PP^1$ of bidegree $(2, 3)$, $\O_C(1, 0)$ is semi-stable.
This can be shown by the same argument from the proof of \cite[Proposition 3.5]{genus_two} by which proves that $\O_C(0, 1)$ is semi-stable.

\medskip

\noindent
(iv) Assume that $\H^0(\F(0, -1)) \neq 0$.
As at (iii) above we can show that $\F$ is a non-split extension of $\CC_p$ by $\O_C(0, 1)$.
Combining the resolutions
\[
0 \lra \O(-2, -2) \lra \O(0, 1) \lra \O_C(0, 1) \lra 0
\]
and
\[
0 \lra \O(-2, -2) \lra \O(-2, -1) \oplus \O(-1, -2) \lra \O(-1, -1) \lra \CC_p \lra 0
\]
we obtain the resolution
\[
0 \to \O(-2, -2) \to \O(-2, -2) \oplus \O(-2, -1) \oplus \O(-1, -2) \to \O(-1, -1) \oplus \O(0, 1) \to \F \to 0.
\]
The map $\O(-2, -2) \to \O(-2, -2)$ is non-zero, otherwise the extension of $\CC_p$ by $\O_C(0, 1)$ would split
(consult the proof of \cite[Proposition 2.3.2]{illinois}).
Thus, we obtain resolution (\ref{M_1}).

Conversely, we assume that $\F$ is given by resolution (\ref{M_1}) and we need to show that $\F$ is semi-stable.
Note first that $\H^0(\F)$ generates a subsheaf of $\F$ of the form $\O_C(0, 1)$.
Assume that $\F$ had a destabilizing subsheaf $\E$. Without loss of generality we may take $\E$ to be semi-stable.
Then $\H^0(\E) \neq \H^0(\F)$, otherwise $\E \simeq \O_C(0, 1)$, which does not destabilize $\F$.
Thus, $\H^0(\E) \simeq \CC$ and $\chi(\E) = 1$.
According to \cite[Corollary 3.4]{genus_two}, $\E$ cannot have Hilbert polynomial $2m + 1$ or $2n + 1$.
Thus, $P_{\E} = n + 1$, $m + 1$ or $m + n + 1$.
If $P_{\E} = n + 1$, then resolution (\ref{(r,1)}) with $r = 0$ fits into a commutative diagram
\[
\xymatrix
{
0 \ar[r] & \O(-1, 0) \ar[r] \ar[d]^-{\beta} & \O \ar[r] \ar[d]^-{\alpha} & \E \ar[r] \ar[d] & 0 \\
0 \ar[r] & \O(-2, -1) \oplus \O(-1, -2) \ar[r] & \O(-1, -1) \oplus \O(0, 1) \ar[r] & \F \ar[r] & 0
}
\]
with $\alpha \neq 0$. Thus $\alpha$ is injective, hence $\beta$ is injective, too, which is absurd.
If $P_{\E} = m+1$ or $m + n + 1$ we get the same contradiction using resolution (\ref{(1,s)}) with $s = 0$, respectively, with $s = 1$.
\end{proof}

\noindent
Let $\MM_2 \subset \MM$ be the subset of sheaves from Proposition \ref{vanishing}(iii). Clearly, $\MM_2 \simeq \PP^{11}$.
Let $\MM_1 \subset \MM$ be the subset of sheaves from Proposition \ref{vanishing}(iv).
Clearly, $\MM_1$ is isomorphic to the universal curve of bidegree $(2, 3)$ in $\PP^1 \times \PP^1$,
so $\MM_1$ has codimension $1$.
Let $\MM_0 \subset \MM$ be the open subset of sheaves $\F$ for which $\H^0(\F(-1, 0)) = 0$ and $\H^0(\F(0, -1)) = 0$.

\begin{proposition}
\label{sheaves_M_0}
The sheaves in $\MM_0$ are precisely the sheaves $\F$ having a resolution of the form
\begin{equation}
\label{M_0}
0 \lra \O(-1, -2) \oplus \O(-1, -1) \overset{\f}{\lra} 2\O \lra \F \lra 0
\end{equation}
with $\f_{12}$ and $\f_{22}$ linearly independent.
\end{proposition}

\begin{proof}
Assume that $\F$ gives a point in $\MM_0$. From the exact sequence (\ref{E_1^{-1,0}}) we get $E_1^{-1, 0} = 0$.
In view of Proposition \ref{vanishing}, display diagram (\ref{E_1}) takes the form
\[
\xymatrix
{
3\O(-1, -1) \ar[r]^-{\f_1} & E_1^{-1, 1} & 0 \\
0 & 0 & 2\O
}
\]
hence exact sequence (\ref{convergence}) becomes
\[
0 \lra \Ker(\f_1) \lra 2\O \lra \F \lra \Coker(\f_1) \lra 0.
\]
From this exact sequence we can compute the Hilbert polynomial of $E_1^{-1, 1}$:
\begin{align*}
P_{E_1^{-1, 1}} & = P_{\Image(\f_1)} + P_{\Coker(\f_1)} \\
& = P_{3\O(-1, -1)} - P_{\Ker(\f_1)} + P_{\F} - P_{2\O} + P_{\Ker(\f_1)} \\
& = P_{3\O(-1, -1)} + P_{\F} - P_{2\O} \\
& = 3mn + 3m + 2n + 2 - 2(m + 1)(n + 1) \\
& = mn + m.
\end{align*}
The exact sequence (\ref{E_1^{-1,1}}) takes the form
\[
0 \lra E_1^{-1, 1} \lra \O(-1, 0).
\]
Since $P_{E_1^{-1, 1}} = P_{\O(-1, 0)}$ it follows that $E_1^{-1, 1} \simeq \O(-1, 0)$.
If $\f_1 = (1 \tensor u, 0 , 0)$, then $\Coker(\f_1)$ has slope zero, so it is a destabilizing quotient sheaf of $\F$.
Thus, $\f_1 = (1 \tensor z, 1 \tensor w, 0)$, $\Ker(\f_1) \simeq \O(-1, -2) \oplus \O(-1, -1)$ and the exact sequence (\ref{convergence})
yields resolution (\ref{M_0}).

Conversely, we assume that $\F$ is given by resolution (\ref{M_0}) and we need to show that $\F$ is semi-stable.
Assume that $\F$ had a destabilizing subsheaf $\E$. Since $\H^0(\F)$ generates $\F$, $\H^0(\E) \simeq \CC$, hence $\chi(\E) = 1$.
As in the proof of Proposition \ref{vanishing}(iv), we have a commutative diagram with exact rows
\[
\xymatrix
{
0 \ar[r] & {\mathcal B} \ar[r] \ar[d]^-{\beta} & \O \ar[r] \ar[d]^-{\alpha} & \E \ar[r] \ar[d] & 0 \\
0 \ar[r] & \O(-1, -2) \oplus \O(-1, -1) \ar[r] & 2\O \ar[r] & \F \ar[r] & 0
}
\]
in which ${\mathcal B}$ is one of the sheaves $\O(-1, 0)$, $\O(0, -1)$, or $\O(-1, -1)$,
and in which $\alpha \neq 0$. Thus $\alpha$ is injective, hence $\beta$ is injective, too.
If ${\mathcal B} = \O(-1, 0)$ or $\O(0, -1)$, then $\beta$ cannot be injective, so we get a contradiction.
If ${\mathcal B} = \O(-1, -1)$, then the hypothesis that $\f_{12}$ and $\f_{22}$ be linearly independent gets contradicted.
\end{proof}


\section{The Homology of $\MM$}
\label{homology}

Let $W \subset \Hom(\O(-1, -2) \oplus \O(-1, -1), 2\O)$ be the subset of morphisms $\f$ for which $\f_{12}$ and $\f_{22}$ are linearly independent
and such that there is no $u \in V_2^*$ satisfying $\f_{11} = \f_{12} (1 \tensor u)$, $\f_{21} = \f_{22} (1 \tensor u)$.
Consider the algebraic group
\[
G = \big( \Aut(\O(-1, -2) \oplus \O(-1, -1)) \times \Aut(2\O) \big)/\CC^*
\]
acting on $W$ by conjugation. Let $W_0 \subset W$ be the open invariant subset of injective morphisms.
Note that $W_0$ is the set of morphisms $\f$ occurring at resolution (\ref{M_0}).

\begin{proposition}
\label{quotient_M_0}
There exists a geometric quotient $W/G$, which is a fiber bundle with fiber $\PP^9$ and base 
the Grassmann variety $\Grass(2, V_1^* \tensor V_2^*)$.
Thus, there exists $W_0/G$ as a proper open subset of $W/G$.
Moreover, $W_0/G_0$ is isomorphic to $\MM_0$.
\end{proposition}

\noindent
The proof of this proposition is analogous to the proof of \cite[Proposition 3.2.2]{illinois}.
In point of fact, $W/G$ is isomorphic to the projectivization of the bundle ${\mathcal U}$ over $\Grass(2, 4)$
introduced at \cite[Section 2.1]{chung_moon}. A natural birational map from $\M(3m + 2n + 1)$ to $\PP({\mathcal U})$
is constructed in \cite{chung_moon}, hence we have a natural birational map from $\M(3m + 2n + 1)$ to $\MM$.

\begin{corollary}
\label{rational}
The variety $\MM$ is rational.
\end{corollary}

\begin{proposition}
\label{complement}
The complement $X$ of $W_0/G$ in $W/G$ consists of two disjoint irreducible components
$X_1 \simeq \PP^1$ and $X_2 \simeq \PP^1 \times \PP^1$.
\end{proposition}

\begin{proof}
Assume that $\f \in W$ and $\det(\f) = 0$. We will examine several cases.

\medskip

\noindent
\emph{Case 1: Assume that $\f \nsim \psi$, where $\psi_{12}$ is a pure tensor.} Then we may write
\[
\f = \left[
\begin{array}{cl}
x \tensor \alpha_1 + y \tensor \alpha_2 & x \tensor z + y \tensor w \\
x \tensor \beta_1 + y \tensor \beta_2 & x \tensor w + y \tensor (az + bw)
\end{array}
\right]
\]
for some $\alpha_1, \alpha_2, \beta_1, \beta_2 \in \Sym^2 V^*_2$, $a, b \in \CC$, $a \neq 0$.
We have
\begin{align*}
\det(\f) = & x^2 \tensor \alpha_1 w + xy \tensor \alpha_1 (az + bw) + xy \tensor \alpha_2 w + y^2 \tensor \alpha_2 (az + bw) \\
& - x^2 \tensor \beta_1 z - xy \tensor \beta_2 z - xy \tensor \beta_1 w - y^2 \tensor \beta_2 w.
\end{align*}
From the relation $\det(\f) = 0$ we get the relations
\begin{align*}
0 & = \alpha_1 w - \beta_1 z, \text{ hence } \alpha_1 = u_1 z,\ \beta_1 = u_1 w \text{ for some } u_1 \in V_2^*, \\
0 & = \alpha_2 (az + bw) - \beta_2 w, \text{ hence } \alpha_2 = u_2 w,\ \beta_2 = u_2 (az + bw) \text{ for some } u_2 \in V_2^*, \\
0 & = \alpha_1 (az + bw) + \alpha_2 w - \beta_2 z - \beta_1 w, \text{ hence} \\
0 & = u_1 z (az + bw) + u_2 w^2 - u_2 z (az + bw) - u_1 w^2, \text{ hence} \\
0 & = (u_1 - u_2)(a z^2 + bzw - w^2), \text{ and hence } u_1 = u_2 = u.
\end{align*}
Thus,
\[
\f = \left[
\begin{array}{ll}
x \tensor zu + y \tensor wu & x \tensor z + y \tensor w \\
x \tensor wu + y \tensor (az + bw)u & x \tensor w + y \tensor (az + bw)
\end{array}
\right] = \left[
\begin{array}{cc}
\f_{12} (1 \tensor u) & \f_{12} \\
\f_{22} (1 \tensor u) & \f_{22}
\end{array}
\right].
\]
This contradicts the choice of $\f \in W$. In Case 1, every $\f \in W$ belongs also to $W_0$.

\medskip

\noindent
\emph{Case 2: Assume that $\f_{12}$ is a pure tensor but $\f \nsim \psi$, where both $\psi_{12}$ and $\psi_{22}$ are pure tensors.}
Then we may write
\[
\f = \left[
\begin{array}{cl}
x \tensor \alpha_1 + y \tensor \alpha_2  & x \tensor z \\
x \tensor \beta_1 + y \tensor \beta_2 & x \tensor w + y \tensor (z + aw)
\end{array}
\right].
\]
We have
\begin{align*}
\det(\f) = & x^2 \tensor \alpha_1 w + xy \tensor \alpha_1 (z + aw) + xy \tensor \alpha_2 w + y^2 \tensor \alpha_2 (z + aw) \\
& - x^2 \tensor \beta_1 z - xy \tensor \beta_2 z.
\end{align*}
From the relation $\det(\f) = 0$ we get the relations
\begin{align*}
\alpha_2 & = 0, \\
\alpha_1 w - \beta_1 z & = 0, \text{ hence } \alpha_1 = uz, \ \beta_1 = uw \text{ for some } u \in V_2^*, \\
\alpha_1 (z + aw) - \beta_2 z & = 0, \text{ hence } \beta_2 = u(z + aw).
\end{align*}
Thus,
\[
\f = \left[
\begin{array}{ll}
x \tensor zu & x \tensor z \\
x \tensor wu + y \tensor (z + aw)u & x \tensor w + y \tensor (z + aw)
\end{array}
\right] = \left[
\begin{array}{cc}
\f_{12}(1 \tensor u) & \f_{12} \\
\f_{22}(1 \tensor u) & \f_{22}
\end{array}
\right].
\]
This contradicts the choice of $\f \in W$. In Case 2, every $\f \in W$ belongs also to $W_0$.

\medskip

\noindent
\emph{Case 3: Assume that $\f_{12} = v_1 \tensor v_2$, $\f_{22} = v_1' \tensor v_2'$ with $\{ v_1^{}, v_1' \}$ linearly independent,
$\{ v_2^{}, v_2' \}$ linearly independent.} Then we may write
\[
\f = \left[
\begin{array}{cc}
x \tensor \alpha_1 + y \tensor \alpha_2 & x \tensor z \\
x \tensor \beta_1 + y \tensor \beta_2 & y \tensor w
\end{array}
\right].
\]
We have
\[
\det(\f) = xy \tensor \alpha_1 w + y^2 \tensor \alpha_2 w - x^2 \tensor \beta_1 z - xy \tensor \beta_2 z.
\]
From the relation $\det(\f) = 0$ we get the relations
\begin{align*}
\alpha_2 & = 0, \\
\beta_1 & = 0, \\
\alpha_1 w - \beta_2 z & = 0, \text{ hence } \alpha_1 = uz, \ \beta_2 = uw \text{ for some } u \in V_2^*.
\end{align*}
Thus,
\[
\f = \left[
\begin{array}{cc}
x \tensor zu & x \tensor z \\
y \tensor wu & y \tensor w
\end{array}
\right] = \left[
\begin{array}{cc}
\f_{12}(1 \tensor u) & \f_{12} \\
\f_{22}(1 \tensor u) & \f_{22}
\end{array}
\right].
\]
This contradicts the choice of $\f \in W$. In Case 3, every $\f \in W$ belongs also to $W_0$.

\medskip

\noindent
\emph{Case 4: Assume that $\f_{12} = v_1 \tensor v$, $\f_{22} = v_1' \tensor v$ with $v \neq 0$ and with $\{ v_1^{}, v_1' \}$ linearly independent.}
Such a morphism $\f$ can be written in the canonical form
\[
\f = \left[
\begin{array}{cc}
x \tensor \alpha_1 + y \tensor \alpha_2 & x \tensor z \\
x \tensor \beta_1 + y \tensor \beta_2 & y \tensor z
\end{array}
\right].
\]
We have
\[
\det(\f) = xy \tensor \alpha_1 z + y^2 \tensor \alpha_2 z - x^2 \tensor \beta_1 z - xy \tensor \beta_2 z.
\]
From the relation $\det(\f) = 0$ we get the relations
\begin{align*}
\alpha_2 & = 0, \\
\beta_1 & = 0, \\
\alpha_1 z - \beta_2 z & = 0, \text{ hence } \alpha_1 = \beta_2 = \alpha.
\end{align*}
Thus,
\[
\f = \left[
\begin{array}{cc}
x \tensor \alpha & x \tensor z \\
y \tensor \alpha & y \tensor z
\end{array}
\right] \sim \left[
\begin{array}{cc}
x \tensor w^2 & x \tensor z \\
y \tensor w^2 & y \tensor z
\end{array}
\right].
\]
The fiber of $X$ over $\spann \{ \f_{12}, \f_{22} \} \in \Grass(2, V_1^* \tensor V_2^*)$ consists of a single point.
The subset
\[
\{ \spann \{ x \tensor v, y \tensor v \},\ v \in V_2^* \setminus \{ 0 \} \} \subset \Grass(2, V_1^* \tensor V_2^*)
\]
is isomorphic to $\PP(V_2^*) \simeq \PP^1$. Thus, we obtain an irreducible component $X_1$ of $X$ isomorphic to $\PP^1$.

\medskip

\noindent
\emph{Case 5: Assume that $\f_{12} = v \tensor v_2$, $\f_{22} = v \tensor v_2'$ with $v \neq 0$ and with $\{ v_2^{}, v_2' \}$ linearly independent.}
Such a morphism $\f$ can be written in the canonical form
\[
\f = \left[
\begin{array}{cc}
x \tensor \alpha_1 + y \tensor \alpha_2 & x \tensor z \\
x \tensor \beta_1 + y \tensor \beta_2 & x \tensor w
\end{array}
\right].
\]
We have
\[
\det(\f) = x^2 \tensor \alpha_1 w + xy \tensor \alpha_2 w - x^2 \tensor \beta_1 z - xy \tensor \beta_2 z.
\]
From the relation $\det(\f) = 0$ we get the relations
\begin{align*}
\alpha_1 w - \beta_1 z = 0, \text{ hence } \alpha_1 = u_1 z,\ \beta_1 = u_1 w \text{ for some } u_1 \in V_2^*, \\
\alpha_2 w - \beta_2 z = 0, \text{ hence } \alpha_2 = u_2 z,\ \beta_2 = u_2 w \text{ for some } u_2 \in V_2^*.
\end{align*}
Thus,
\[
\f = \left[
\begin{array}{cc}
x \tensor u_1 z + y \tensor u_2 z & x \tensor z \\
x \tensor u_1 w + y \tensor u_2 w & x \tensor w
\end{array}
\right] \sim \left[
\begin{array}{cc}
y \tensor u_2 z & x \tensor z \\
y \tensor u_2 w & x \tensor w
\end{array}
\right].
\]
The fiber of $X$ over $\spann \{ \f_{12}, \f_{22} \} \in \Grass(2, V_1^* \tensor V_2^*)$ is parametrized by $u_2$ so it is isomorphic to
$\PP(V_2^*) \simeq \PP^1$. The subset
\[
\{ \spann \{ v \tensor z, v \tensor w \},\ v \in V_1^* \setminus \{ 0 \} \} \subset \Grass(2, V_1^* \tensor V_2^*)
\]
is isomorphic to $\PP(V_1^*) \simeq \PP^1$. Thus, we obtain an irreducible component $X_2$ of $X$ isomorphic to $\PP^1 \times \PP^1$.
\end{proof}

\noindent
{\bf Proof of Theorem \ref{poincare_polynomial}.}
We have
\begin{align*}
\Poly(\MM) & = \Poly(\MM_0) + \Poly(\MM_1) + \Poly(\MM_2) \\
& = \Poly(W/G) - \Poly(X_1) - \Poly(X_2) + \Poly(\MM_1) + \Poly(\MM_2) \\
& = \Poly(\PP^9) \Poly(\Grass(2, 4)) - \Poly(\PP^1) - \Poly(\PP^1 \times \PP^1) + \Poly(\PP^{10}) \Poly(\PP^1 \times \PP^1) + \Poly(\PP^{11}) \\
& = \frac{\xi^{10} - 1}{\xi - 1}(\xi^4 + \xi^3 + 2\xi^2 + \xi + 1) - (\xi + 1) - (\xi + 1)^2 + \frac{\xi^{11} - 1}{\xi - 1}(\xi + 1)^2 \\
& \qquad + \frac{\xi^{12} - 1}{\xi - 1}.
\end{align*}

\end{document}